\documentclass[12pt]{amsart}
\usepackage{amsfonts}
\usepackage{amsmath}
\usepackage{amssymb}
\usepackage[margin=1.2in]{geometry}

\newtheorem{theorem}{Theorem}[section]

\newtheorem{corollary}[theorem]{Corollary}
\newtheorem{lemma}[theorem]{Lemma}

\theoremstyle{definition}

\newtheorem{example}[theorem]{Example}

\numberwithin{equation}{section}

\begin{document}
\title[Asymptotic density of Beurling generalized integers]{On Diamond's $L^1$ criterion for asymptotic density of Beurling generalized integers}

\author[G. Debruyne]{Gregory Debruyne}
\thanks{G. Debruyne gratefully acknowledges support by Ghent University, through a BOF Ph.D. grant}
\address{Department of Mathematics: Analysis, Logic and Discrete Mathematics\\ Ghent University\\ Krijgslaan 281\\ 9000 Ghent\\ Belgium}
\email{gregory.debruyne@UGent.be}

\author[J. Vindas]{Jasson Vindas}
\thanks{The work of J. Vindas was supported by the Research Foundation--Flanders, through the FWO-grant number 1520515N}
\address{Department of Mathematics: Analysis, Logic and Discrete Mathematics\\ Ghent University\\ Krijgslaan 281\\  9000 Ghent\\ Belgium}
\email{jasson.vindas@UGent.be}



\begin{abstract} We give a short proof of the $L^{1}$ criterion for Beurling generalized integers to have a positive asymptotic density. We actually prove the existence of density under a weaker hypothesis. We also discuss related sufficient conditions for the estimate $m(x)=\sum_{n_{k}\leq x} \mu(n_k)/n_k=o(1)$, with $\mu$ the Beurling analog of the M\"{o}bius function.
\end{abstract}

\maketitle

\section{Introduction}
Let $\{p_{k}\}_{k=1}^{\infty}$ be a Beurling generalized prime number system, that is, an unbounded sequence of real numbers $p_{1}\leq p_{2}\leq p_3\leq  \dots$ subject to the only requirement $p_{1}>1$. Its associated set of generalized integers \cite{beurling,diamond-zhangbook} is the multiplicative semigroup generated by the generalized primes and 1. We arrange them in a non-decreasing sequence where multiplicities are taken into account, $1=n_0<n_{1}\leq n_{2}\leq \dots $. One then considers the counting functions of the generalized integers and primes,
\begin{equation*}
N(x)=\sum_{n_{k}\leq x}1 \quad \mbox{ and }\quad \pi(x)=\sum_{p_{k}\leq x}1.
\end{equation*}

A central question in the theory of generalized numbers is to determine conditions, as minimal as possible, on one of the functions $N(x)$ or $\pi(x)$ such that the other one becomes close to its classical counterpart. Starting with the seminal work of Beurling \cite{beurling}, the problem of finding requirements on $N(x)$ that ensure the validity of the prime number theorem $\pi(x)\sim x/\log x$ has been extensively investigated; see, for example, \cite{beurling,diamond-zhangbook,kahane1997,SP-V2012,zhang2015}. In the opposite direction, Diamond  proved in 1977 \cite{Diamond1977} the following important $L^1$ criterion for generalized integers to have a positive density.
As in the classical case, we denote
\begin{equation*}
 \Pi(x)=\pi(x)+\frac{1}{2}\pi(x^{1/2})+\frac{1}{3}\pi(x^{1/3})+\dots\: .
\end{equation*}
\begin{theorem}\label{gith1}
Suppose that 
\begin{equation}
\label{gieq1.1}
\int_{2}^{\infty}\left|\Pi(x)-\frac{x}{\log x}\right|\:\frac{\mathrm{d}x}{x^{2}}<\infty.
\end{equation}
Then, there is $a>0$ such that
\begin{equation}
\label{gieq1.2}
N(x)\sim ax.
\end{equation}
\end{theorem}
It can be shown (see \cite[Thm. 5.10 and Lemma 5.11, pp. 47--48]{diamond-zhangbook}) that the value of the constant $a$ in (\ref{gieq1.2}) is given by
\[
\log a= \int_{1}^{\infty} x^{-1}\left(\mathrm{d}\Pi(x)-\frac{1-1/x}{\log x}\: \mathrm{d}x\right).
\]

Diamond's proof of Theorem \ref{gith1} is rather involved. It depends upon subtle decompositions of the measure $\mathrm{d}\Pi$ and then an iterative procedure. In their recent book \cite[p. 76]{diamond-zhangbook}, Diamond and Zhang have asked whether there is a simpler proof of this theorem. 

 The goal of this article is to provide a short proof of Theorem \ref{gith1}. Our proof is of Tauberian character. It is based on the analysis of the boundary behavior of the zeta function 
 $$
 \zeta(s)=\int_{1^{-}}^{\infty}x^{-s}\mathrm{d}N(x)
 $$
via local pseudofunction boundary behavior and then an application of the distributional version of the Wiener-Ikehara theorem \cite{Debruyne-VindasWiener-Ikehara,korevaar2005}. Our method actually yields \eqref{gieq1.2} under a weaker hypothesis than \eqref{gieq1.1}, see Theorem \ref{gith2} in Section \ref{proof Diamond's theorem}. 
We mention that Kahane has recently obtained another different proof yet of Theorem \ref{gith1} in \cite{kahane2017}.

In Section \ref{section gi Mobius}, we apply our Tauberian approach to study the estimate
\begin{equation}
\label{gieq1.3}
m(x)=\sum_{n_{k}\leq x}\frac{\mu(n_{k})}{n_{k}}=o(1),
\end{equation}
with $\mu$ the Beurling analog of the M\"{o}bius function. The sufficient conditions we find here for \eqref{gieq1.3} generalize the ensuing recent result of Kahane and Sa\"{i}as \cite{Kahane-Saias2016a,Kahane-Saias2016b}: the $L^1$ hypothesis (\ref{gieq1.1}) suffices for the estimate \eqref{gieq1.3}.

\section{Tauberian machinery}
\label{section preli}
We collect in this section some Tauberian theorems that play a role in the article. These Tauberian theorems are in terms of local pseudofunction boundary behavior \cite{Debruyne-VindasWiener-Ikehara, Debruyne-VindasComplexTauberians,  korevaar2005,SP-V2012}, which turns out to be an optimal assumption for many complex Tauberian theorems, in the sense that it often leads to ``if and only if''  results. 

We normalize Fourier transforms as $\hat{\varphi}(t)=\mathcal{F}\{\varphi;t\}=\int_{-\infty}^{\infty}e^{-itx}\varphi(x)\:\mathrm{d}x$, and  interpret them in the sense of tempered distributions when the integral definition does not make sense. The standard Schwartz test function spaces of compactly supported smooth functions (on an open subset $U\subseteq \mathbb{R}$)  and rapidly decreasing smooth functions are denoted by $\mathcal{D}(U)$ and $\mathcal{S}(\mathbb{R})$, while $\mathcal{D}'(U)$ and $\mathcal{S}'(\mathbb{R})$ stand for their topological duals, the spaces of distributions and tempered distributions \cite{bremermann}.  We write $\langle f,\varphi\rangle$, or $\langle f(x),\varphi(x)\rangle$ with the use of a dummy variable of evaluation,  for the dual pairing between a distribution $f$ and a test function $\varphi$; as usual, locally integrable functions are regarded as distributions via $\langle f(x),\varphi(x)\rangle=\int_{-\infty}^{\infty}f(x)\varphi(x)\mathrm{d}x$.  

Denote as $A(\mathbb{R})=\mathcal{F}(L^{1}(\mathbb{R}))$ the Wiener algebra, its dual $PM(\mathbb{R})=\mathcal{F}(L^{\infty}(\mathbb{R}))$ is the space of global pseudomeasures. We call $f\in PM(\mathbb{R})$ a global pseudofunction if additionally $\lim_{|x|\to\infty} \hat{f}(x)=0$, and write $f\in PF(\mathbb{R})$. A Schwartz distribution $g\in\mathcal{D}'(U)$ is said to be a local pseudofunction on an open set $U$ if every point of $U$ has a neighborhood where $g$ coincides with a global pseudofunction; we then write $g\in PF_{loc}(U)$.  Equivalently, the latter holds if and only if  $\lim_{|x|\to\infty} \widehat{\varphi g}(x)=0$ for every $\varphi\in \mathcal{D}(U)$. One defines similarly the local Wiener algebra $A_{loc}(U)$ of continuous functions. Note that ${A}_{loc}(U)$ is an algebra under pointwise multiplication, while ${PF}_{loc}(U)$ has a natural ${A}_{loc}(U)$-module structure. Since $C^{\infty}(U)\subset {A}_{loc}(U)$, we obtain that smooth functions are multipliers for ${A}_{loc}(U)$ and ${PF}_{loc}(U)$. Also, $L^{1}_{loc}(U)\subsetneq PF_{loc}(U)$, in view of the Riemann-Lebesgue lemma.

Let $G(s)$ be analytic on the half-plane $\Re e\:s>1$ and let $U\subset \mathbb{R}$ be open. We say that $G$ has local pseudofunction boundary behavior on the boundary open subset $1+iU$ if  $G$ admits a local pseudofunction as distributional boundary value on $1+i U$, that is, if there is $g\in PF_{loc}(U)$ such that 
\begin{equation*}
\lim_{\sigma\to1^{+}}\int_{-\infty}^{\infty}G(\sigma+it)\varphi(t)\mathrm{d}t=\left\langle g(t),\varphi(t)\right\rangle, \quad \mbox{for each } \varphi\in\mathcal{D}(U).
\end{equation*}
If $U=\mathbb{R}$, we say that $G$ has local pseudofunction boundary behavior on $\Re e\: s=1$. We often write $g(t)=G(1+it)$ for its boundary value distribution.
Likewise, one defines boundary behavior with respect to other spaces such as $A_{loc}$ or $L^1_{loc}$. We emphasize that $L^1_{loc}$-boundary behavior, continuous, or analytic extension are very special cases of local pseudofunction boundary behavior.

We will employ the following distributional version of the Wiener-Ikehara theorem, due to Korevaar \cite{korevaar2005}. (See \cite{Debruyne-VindasWiener-Ikehara,Debruyne-VindasComplexTauberians} for more general results.)

\begin{theorem} \label{W-IMellinth1} Let $S$ be a non-decreasing function having support in $[1,\infty)$. Then, 
$$S(x)\sim ax\: $$
 if and only if its Mellin-Stieltjes transform $\int^{\infty}_{1^{-}} x^{-s} \mathrm{d}S(x)$ converges for $\Re e \: s > 1$
and 
\begin{equation*}
\int^{\infty}_{1^{-}} x^{-s} \mathrm{d}S(x)- \frac{a}{s-1}
\end{equation*}
 admits local pseudofunction boundary behavior on the line $\Re e \: s = 1$.   
\end{theorem} 

The next Tauberian theorem is a recent extension of the Ingham-Fatou-Riesz theorem, obtained by the authors in \cite{Debruyne-VindasComplexTauberians}. A function $\tau$ is called \emph{slowly decreasing} (with multiplicative arguments) if for each $\varepsilon > 0$ there exists $c>1$ such that 
\[ \liminf_{x\to\infty}\inf_{\eta\in[1,c]}(\tau(\eta x) - \tau(x)) \geq -\varepsilon.
\]

\begin{theorem}
\label{F-RMellinth} Let $\tau\in L^{1}_{loc}[1,\infty)$ be slowly decreasing.
Then,
$$\tau(x) =a\log x+b + o(1)$$
if and only if its Mellin transform $\int^{\infty}_{1} x^{-s} \tau(x)\mathrm{d}x$ converges for $\Re e \: s > 1$
and
\begin{equation*}
\int^{\infty}_{1} x^{-s}\tau(x)\mathrm{d}x - \frac{a}{(s-1)^{2}} -\frac{b}{s-1}  
\end{equation*}
admits local pseudofunction boundary behavior on the line $\Re e \: s = 1.$
\end{theorem}

It is very important to determine sufficient criteria in order to conclude that an analytic function has local pseudofunction boundary behavior. The ensuing lemma provides such a criterion for the product of two analytic functions.

\begin{lemma}\label{gil1}
Let $G$ and $F$ be analytic on the half-plane $\Re e\: s>1$ and let $U$ be an open subset of $\mathbb{R}$. If $F$ has local pseudofunction boundary behavior on $1+iU$ and $G$ has $A_{loc}$-boundary behavior on $1+iU$, then $G\cdot F$ has local pseudofunction boundary behavior on $1+iU$.
\end{lemma}
\begin{proof} Fix a relatively compact open subset $V$ such that $\overline{V}\subset U$. By definition, we can find $g\in L^{1}(\mathbb{R})$ and $f\in L^{\infty}(\mathbb{R})$ such that $\hat{g}(t)= G(1+it)$ and $\hat{f}(t)=F(1+it)$ on $V$ and $\lim_{|x|\to\infty}f(x)=0$.
Let $g_{\pm}(x)=g(x)H(\pm x)$ and $f_{\pm}(x)=f(x)H(\pm x)$, where $H$ is the Heaviside function, i.e., the characteristic function of the interval $[0,\infty)$. We define $G_{\pm}(s)=\mathcal{L}\{g_{\pm};s\}$ and $F_{\pm}(s)=\mathcal{L}\{f_{\pm};s\}$, where $\mathcal{L}$ stands for the Laplace transform so that $G_{+}(s)$ and $F_{+}(s)$ are analytic on $\Re e\:s>0$, whereas $G_{-}(s)$ and $F_{-}(s)$ are defined and analytic on $\Re e\:s<0$. Observe that \cite{bremermann} $\hat{g}_{\pm}(t)=\lim_{\sigma\to0^{+}} G_{\pm}(\pm\sigma+it)$ and $\hat{f}_{\pm}(t)=\lim_{\sigma\to0^{+}}F_{\pm}(\pm\sigma+it)$, where the limit is taken in $\mathcal{S}'(\mathbb{R})$ (in the first case, the limit actually holds uniformly for $t\in \mathbb{R}$ because $g_{\pm}\in L^{1}(\mathbb{R})$). Obviously, we also have $\hat{g}=\hat{g}_{-}+\hat{g}_{+}$ and $\hat{f}=\hat{f}_{-}+\hat{f}_{+}$ on $\mathbb{R}$. Consider the analytic function, defined off the line $1+i\mathbb{R}$,
\[
\tilde{G}(s)=\begin{cases}
G(s)- G_{+}(s-1) & \quad \mbox{if }\Re e\: s>1,
\\
G_{-}(s-1) & \quad \mbox{if }\Re e\: s<1.
\end{cases}
\]
The function $\tilde{G}(s)$ has zero jump across the boundary set $iV+1$, namely,
\[
\lim_{\sigma\to 0^{+}} \tilde{G}(1+\sigma +it)- \tilde{G}(1-\sigma+it)= 0,
\]
where the limit is taken in the distributional sense\footnote{The limit actually holds uniformly for $t$ in compact subsets of $V$, as follows from the next sentence.}. The distributional version of the Painlev\'{e} theorem on
analytic continuation (also known as the edge-of-the-wedge theorem \cite[Thm. B]{rudin1971}) implies that $\tilde{G}$ has analytic continuation through $1+iV$. Exactly the same argument gives that $\tilde{F}(s)=F_{-}(s-1)$ has analytic continuation through $1+iV$ as well and $F(s)=\tilde{F}(s)+F_{+}(s-1)$.
Now, 
$$
G(s)F(s)= \tilde{G}(s)F(s)+ \tilde{F}(s) G_{+}(s-1)+\mathcal{L}\{g_{+}\ast f_{+}; s-1\},
$$
in the intersection of a complex neighborhood of $1+iV$ and the half-plane $\Re e\:s>1$. Taking boundary values on $1+iV$, we obtain that $(G\cdot F)(1+it)=\tilde{G}(1+it) F(1+it)+\tilde{F}(1+it)\hat{g}_{+}(t)+\widehat{f_{+}\ast g_{+}}(t)\in PF_{loc}(V)$, because real analytic functions are multipliers for local pseudofunctions and $\lim_{|x|\to\infty}(f_{+}\ast g_{+})(x)=0$.
\end{proof}

\section{Proof of Theorem \ref{gith1}}\label{proof Diamond's theorem}
Our starting point is the zeta function link between the non-decreasing functions $N$ and $\Pi$, 
\begin{equation}
\label{gieq3.1}
\zeta(s)=\int_{1^{-}}^{\infty}x^{-s}\mathrm{d}N(x)=\exp\left(\int_{1}^{\infty}x^{-s}\mathrm{d}\Pi(x)\right).
\end{equation}
The hypothesis (\ref{gieq1.1}) is clearly equivalent to 
\begin{equation}
\label{gieq3.2}
\int_{2}^{\infty}\left|\Pi(x)-\Pi_{0}(x)\right|\:\frac{\mathrm{d}x}{x^{2}}<\infty,
\end{equation}
where 
\begin{equation}
\label{gieqcanonicalpi}
\Pi_{0}(x)=\int_{1}^{x}\frac{1-1/u}{\log u} \:\mathrm{d}u \quad \mbox{ for } x\geq1.
\end{equation}
Note also that 
\[
\int_{1}^{\infty}x^{-s}\mathrm{d}\Pi_{0}(x)=\log\left( \frac{s}{s-1}\right) \quad \mbox{for }\Re e\:s>1.
\]
This guarantees the convergence of \eqref{gieq3.1} for $\Re e\: s>1$.
Calling
\begin{equation}\label{gieq3.3}
J(s)=\int_{1}^{\infty}x^{-1-s}(\Pi(x)-\Pi_{0}(x))\mathrm{d}x\:, \quad \Re e\:s>1,
\end{equation}
$\log a=J(1)$, and subtracting $a/(s-1)$ from \eqref{gieq3.1}, we obtain the expression
\begin{equation}
\label{gieq3.4}\zeta(s)-\frac{a}{s-1}= \frac{s e^{sJ(1)}-e^{J(1)}}{s-1}+ s^{2}e^{sJ(1)} \cdot \frac{e^{s(J(s)-J(1))}-1}{s(J(s)-J(1))}\cdot \frac{J(s)-J(1)}{s-1}\: .
\end{equation}
The first summand in the right side of \eqref{gieq3.4} and the term $s^{2}e^{sJ(1)}$ are entire functions. Thus, Theorem \ref{W-IMellinth1} yields (\ref{gieq1.2}) if we verify that 
\begin{equation}
\label{gieq3.5} \left(\frac{e^{s(J(s)-J(1))}-1}{s(J(s)-J(1))}\right)\cdot \frac{J(s)-J(1)}{s-1}\: 
\end{equation}
has local pseudofunction boundary behavior on $\Re e\:s=1$. The hypothesis (\ref{gieq3.2}) gives that $J$ extends continuously to $\Re e\: s=1$, but also the more important membership relation $J(1+it)\in A(\mathbb{R})$. Thus, $(1+it)(J(1+it)-J(1))\in A_{loc}(\mathbb{R})$. Since the local Wiener algebra is closed under (left) composition with entire functions, we conclude that the first factor in (\ref{gieq3.5}) has $A_{loc}$-boundary behavior on $\Re e\:s=1$. On the other hand, the second factor of \eqref{gieq3.5} has as boundary distribution on $\Re e\:s=1$ the Fourier transform of $\int_{1}^{e^{y}}u^{-2}(\Pi(u)-\Pi_{0}(u))\mathrm{d}u- J(1)=o(1)$, a global pseudofunction. So, the local pseudofunction boundary behavior of \eqref{gieq3.4} is a consequence of Lemma \ref{gil1}. This establishes Theorem \ref{gith1}.

\section{A generalization of Theorem \ref{gith1}}\label{section generalization Diamond's theorem}
A small adaptation of our method from the previous section applies to show the following generalization of Diamond's theorem: 

If the zeta function \eqref{gieq3.1} converges for $\Re e\:s>1$, then $A_{loc}$-boundary behavior of
\begin{equation}
\label{gieq4.1}
\log \zeta(s)- \log\left(\frac{1}{s-1}\right)
\end{equation}
on $\Re e\:s=1$ still suffices for $N$ to have a positive asymptotic density. Indeed, a key point in Section \ref{proof Diamond's theorem} to establish (\ref{gieq1.2}) via Theorem \ref{W-IMellinth1} was  the $A_{loc}$-boundary behavior on $\Re e \: s=1$ of the function $J$ defined in \eqref{gieq3.3}, but the latter is in fact equivalent to the assumption of $A_{loc}$-boundary behavior on \eqref{gieq4.1}. The function \eqref{gieq3.4} then has local pseudofunction boundary behavior on $\Re e\:s=1$ because
$ (J(s)-J(1))/(s-1)$ does, as follows from Lemma \ref{gil2} below. 

We can actually deduce a more general result. Note that the very last argument in the proof of Theorem \ref{gith2} we give below could also have been used to show Theorem \ref{gith1} in a more direct way through Lemma \ref{gil2}.

\begin{theorem}\label{gith2}
Suppose the zeta function \eqref{gieq3.1} converges for $\Re e\:s>1$ and there are a discrete set of points $0<\eta<t_1<t_2<\dots$ and constants $-1< \beta_1,\beta_2,\dots$ such that
\begin{itemize}
\item[(a)] $\log \zeta(s)- \log(1/(s-1))$ has $A_{loc}$-boundary behavior on $1+i(-\eta,\eta)$,
\item [(b)] for each $T>0$ there is a constant $K_{T}>0$ such that
\begin{equation*}
\log|\zeta (\sigma+it)|\leq K_{T} + \sum_{0<t_n<T}\beta_n\log|\sigma-1+i(t-t_n)|
\end{equation*}
for every $\eta/2<t<T$ and $1<\sigma<2$. 
\end{itemize}
Then, $N$ has a positive asymptotic density. 
\end{theorem}
\begin{proof} Set $G(s)=\exp(\log \zeta(s) - \log (1/(s-1)))$,  $a=G(1)$, and 
$$F_{T}(s)=\log \zeta(s) -\sum_{t_n<T}\beta_n(\log(s-1-it_n)+\log (s-1+it_n)).$$
Condition (b) says that $\Re e \: F_{T}(s)$ is bounded from above on the rectangles $(1,2)\times (\eta/2,T)$ and $(1,2)\times (-T,-\eta/2)$. Thus, since $T$ is arbitrary,
$$
\zeta(s)-\frac{a}{s-1}=  \exp(F_{T}(s)) \prod_{0<t_n<T}((s-1)^{2}+t^2_n)^{\beta_n} - \frac{a}{s-1}
$$
has $L^1_{loc}$-boundary extension to $1+i(\mathbb{R}\setminus [-\eta/2,-\eta/2])$. By condition (a), $G(1+it)\in A_{loc}(-\eta,\eta)$ and
$
\zeta(s)-a/(s-1)$
is equal to
\begin{equation}
\label{gieq4.2}
\frac{G(s)-G(1)}{s-1}\:.
\end{equation}
So, (\ref{gieq1.2}) follows at once by combining Theorem \ref{W-IMellinth1} with the next lemma.
\end{proof}

\begin{lemma}
\label{gil2} Let $G(s)$ be analytic for $\Re e\:s>1$ and let $U\subset \mathbb{R}$ be open. If $G$ has boundary extension to $1+iU$ as an element of the local Wiener algebra $G(1+it)\in A_{loc}(U)$ and $s_0\in 1+iU$, then
$$
\frac{G(s)-G(s_0)}{s-s_0}
$$
 has local pseudofunction boundary behavior on $1+iU$.
\end{lemma}
\begin{proof} We may assume that $0\in U$ and $s_0=1$.
Since (\ref{gieq4.2}) has a continuous boundary function on $1+iU$ except perhaps at $s=1$, it is enough to verify its local pseudofunction boundary behavior at $s=1$. As in the proof of Lemma \ref{gil1} with the aid of the Painlev\'{e} theorem on analytic continuation, we can find an analytic function $\tilde{G}(s)$ in a neighborhood of $s=1$ and a function $g_{+} \in L^{1}(\mathbb{R})$ such that $\operatorname*{supp} g_{+}\subseteq [0,\infty)$ and $G(1+it)=\tilde{G}(1+it)+\hat{g}_{+}(t)$ for, say, $t\in (-\lambda,\lambda)$. The boundary value of (\ref{gieq4.2}) on $1+(-i\lambda,i\lambda)$ is the sum of the analytic function 
$$
\frac{\tilde{G}(1+it)-\tilde{G}(1)}{it}
$$
and the Fourier transform $\hat{f}(t)$, where $f$ is the function $f(x)=-\int_{x}^{\infty} g_{+}(u)\mathrm{d}u$ for $x>0$ and $f(x)=0$ for $x<0$, whence the assertion follows.
\end{proof}

We also obtain,

\begin{corollary}
\label{gic3} Suppose there are $0<t_1<t_2<\dots<t_k$,  $y_1,y_2,\dots,y_k$, and $b_1,b_2,\dots,b_k$ such that
\begin{equation}
\label{gieq4.3}
\int_{2}^{\infty} \left|\Pi(x)-\frac{x}{\log x}\left(1+\sum_{j=1}^{k} b_j \cos(t_j \log x+y_j)\right)\right|\frac{\mathrm{d}x}{x^2}<\infty.
\end{equation}
If 
\begin{equation}
\label{gieq4.4}
b_{j} (1+t^{2}_{j})^{1/2}\cos(y_j+\arctan t_j)<2, \quad j=1,\dots, k,
\end{equation}
then $N$ has positive asymptotic density.
\end{corollary}
\begin{proof}
Indeed, setting $\theta_j=y_j+\arctan t_{j}$, we obtain from (\ref{gieq4.3}) that
$$
\log \zeta(s)+\log (s-1) +\sum_{j=1}^{k} b_j \frac{(1+t_j^2)^{1/2}}{2} (e^{i\theta_j}\log (s-1-it_{j})+e^{-i\theta_j} \log (s-1+it_{j}))
$$ 
has $A_{loc}$-boundary behavior on $\Re e\:s=1$. An application of Theorem \ref{gith2} then yields the result.
\end{proof}

We now discuss three examples. The first two examples show that there are instances of generalized number systems for which the $L^1$ condition \eqref{gieq1.1} may fail, but the other criteria given in this section apply to show $N(x)\sim ax$. The third example shows that the assumption $-1< \beta_1,\beta_2,\dots$ in Theorem \ref{gith2} cannot be relaxed.

\begin{example}
\label{giex1} Let $k\geq2$ be an integer and let $\varphi$ be a (non-trivial) $C^{k}$ function with $\operatorname*{supp}\varphi\subset (0,1)$ and such that $\|\varphi^{(k)}\|_{L^{\infty}(\mathbb{R})}\leq 1/8$. We consider the generalized number system with absolutely continuous prime distribution function
$$
\Pi(x)= \Pi_{0}(x) + x\sum_{n=3}^{\infty}\frac{1}{n\log^{1/k} n}\varphi^{(k-1)}((\log n)^{1/k}(\log x-n)), 
$$
where $\Pi_{0}$ is the function \eqref{gieqcanonicalpi}.
Since $\|\varphi^{(k-1)}\|_{L^{\infty}(\mathbb{R})}\leq 1/8 $ as well, 
$$
\left|\left(x\sum_{n=3}^{\infty}\frac{1}{n(\log n)^{1/k}}\varphi^{(k-1)}((\log n)^{1/k}(\log x-n))\right)'\right| \quad
\begin{cases}
=0 & \mbox{if } x\leq e^3,
\\
\leq\displaystyle \frac{1}{2\log x} & \mbox{if }x> e^{3} ,
\end{cases}
$$
and thus $\Pi'(x)\geq 0$ for all $x\geq1$. Also,
\[
\int_{2}^{\infty}\left|\Pi(x)-\Pi_{0}(x)\right|\:\frac{\mathrm{d}x}{x^{2}}= 
\|\varphi^{(k-1)}\|_{L^{1}(\mathbb{R})} \sum_{n=3}^{\infty}\frac{1}{n(\log n)^{2/k}}=\infty,
\]
so that \eqref{gieq1.1} does not hold for this example. On the other hand, setting
\[
f(u)=\sum_{n=3}^{\infty}\frac{1}{n\log n} \varphi((\log n)^{1/k}(u-n)),
\]
we have that
\[\log \zeta(s)-\log \left(\frac{1}{s-1}\right)=\log s + s(s-1)^{k-1}\int_{0}^{\infty}e^{-(s-1)u}f(u)\:\mathrm{d}u
\]
has $A_{loc}$-boundary behavior on $\Re e\:s=1$ because
$$
\int_{0}^{\infty}|f(u)|\: \mathrm{d}u= \|\varphi\|_{L^1(\mathbb{R})} \sum_{n=3}^{\infty}\frac{1}{n(\log n)^{1+1/k} }<\infty.
$$
So, Theorem \ref{gith2} applies to show $N(x)\sim ax$ for some $a>0$.
\end{example}

\begin{example}
\label{giex2}
Let 
\[
\mathrm{d}\Pi(u)=\frac{1+\cos(\log u)}{\log u} \chi_{[2,\infty)}(u)\mathrm{d}u.
\]
This continuous generalized number system is a modification of the one used by Beurling to show the sharpness of his PNT \cite{beurling}. Note the PNT fails for $\Pi$, one has instead
\begin{equation}
\label{gieq4.5}
\Pi(x)=\frac{x}{\log x}\left(1+\frac{\sqrt{2}}{2}\cos \left(\log x-\frac{\pi}{4}\right)\right)+O\left(\frac{x}{\log^{2} x}\right).
\end{equation}
It is then clear that
$$
\int_{2}^{\infty}\left|\Pi(x)-\frac{x}{\log x}\right|\: \frac{\mathrm{d}{x}}{x^2}=\infty,
$$
but $\Pi$ satisfies the hypotheses of Corollary \ref{gic3}, so that $N(x)\sim ax$ holds.
\end{example}

\begin{example}
\label{giex3}
We consider
$$
\Pi (x)=\sum_{2^{k+1/2}\leq x} \frac{2^{k+1/2}}{k}\,.
$$
Its zeta function is $4\pi i/\log 2$ periodic and in fact given by 
$$
\log \zeta(s) = - 2^{-(s-1)/2}\log(1-2^{-(s-1)}).
$$
From this explicit formula one verifies that conditions (a) and (b) from Theorem \ref{gith2} are satisfied with 
$\eta= 2\pi/\log 2$, $t_{n}=4\pi n/\log 2$, and $\beta_n=-1$, $n\in \mathbb{N}_{+}$. On the other hand, it has been proved in \cite[Ex. 4.2]{Debruyne-Diamond-Vindas} that $N(x)$ has no asymptotic density due to wobble,
\[
\liminf_{x \to \infty} \frac{N(x)}{x} < 1.37< 1.52 < \limsup_{x \to \infty} \frac{N(x)}{x},
\]
and moreover $m(x)=\Omega(1)$. 
\end{example}

\section{The estimate $m(x)=o(1)$}\label{section gi Mobius}

In this section we study sufficient conditions that imply the estimate (\ref{gieq1.3}). As is actually the case in the previous sections, it is not essential to assume that the generalized number system is discrete; indeed, what is important is that $N$ and $\Pi$ are non-decreasing and satisfy (\ref{gieq3.1}). The measure $\mathrm{d}M$ denotes the (multiplicative) convolution inverse of $\mathrm{d}N$, or equivalently, in terms of its Mellin transform,
$$
\int_{1^{-}}^{\infty} x^{-s}\mathrm{d}M(x)=\frac{1}{\zeta(s)}\:.
$$
The estimate \eqref{gieq1.3} then takes the form
\[
m(x)=\int_{1^{-}}^{x}\frac{\mathrm{d}M(t)}{t}= o(1).
\]
\begin{theorem}
\label{gith3}
Suppose that $N$ has asymptotic density \eqref{gieq1.2}
and there are a discrete set of points $0<t_1<t_2<\dots$ and a corresponding set of constants $\beta_1,\beta_2,\dots<1$ such that for each $T>0$ there is $K_{T}>0$ such that
\begin{equation}
\label{gieq5.2}
\log|\zeta (\sigma+it)|>  -K_{T} -\log |\sigma-1+it|+ \sum_{0<t_n<T}\beta_n\log|\sigma-1+i(t-t_n)|
\end{equation}
holds for every $0<t<T$ and $1<\sigma<2$. Then, $m(x)=o(1)$ holds. 
\end{theorem}

\begin{proof} First observe that
 $m(x)$ is slowly oscillating (in multiplicative sense),
\begin{align*}
\sup_{\eta\in[1,c]}|m(\eta x)-m(x)|&\leq \int_{x}^{cx}\frac{|\mathrm{d}M(u)|}{u}\leq \int_{x}^{cx}\frac{\mathrm{d}N(u)}{u}
\\
&
= \frac{N(cx)}{cx}-\frac{N(x)}{x}+\int_{x}^{cx}\frac{N(u)}{u^{2}}\:\mathrm{d}u
\\
&
= o(1)+a\log c, \quad x\to\infty.
\end{align*}
Writing 
\[
I_{T}(s)=\log \zeta(s)-\log(1/(s-1))-\sum_{0<t_n<T} \beta_n(\log (s- 1-it_n)+\log(s-1+it_n)),
\]
 the condition \eqref{gieq5.2} tells that $-\Re e\: I_{T}(\sigma+it)$ is bounded from above for $1<\sigma<2$ when $t$ stays on the interval $(-T,T)$. Now,
\[
\int_{1}^{\infty}\frac{m(x)}{x^{s}}\: \mathrm{d}x= \frac{1}{(s-1)\zeta(s)}= \exp(-I_{T}(s)) \prod_{0<t_n<T}((s-1)^{2}+t^2_n)^{-\beta_n} 
\]
has $L^{1}_{loc}$-boundary extension to $\Re e\: s=1$ and hence, by Theorem \ref{F-RMellinth}, $m(x)=o(1)$.
\end{proof}

A somewhat simpler sufficient condition for $m(x)=o(1)$, included in Theorem \ref{gith3}, is stated in the next corollary. 
\begin{corollary}
\label{gic1} If $N(x)\sim a x$ and $\log \zeta(s)-\log(1/(s-1))$ has a continuous extension to $\Re e\:s=1$, then the estimate $m(x)=o(1)$ holds.
\end{corollary}

In view of Theorem \ref{gith2}, the hypotheses of Corollary \ref{gic1} are fulfilled if $\log \zeta(s)-\log(1/(s-1))$ has $A_{loc}$-boundary behavior on $\Re e\:s=1$. In particular,

\begin{corollary}
\label{gic2} 
The condition \eqref{gieq1.1} implies the estimate $m(x)=o(1)$.
\end{corollary}
In addition, a less restrictive set of hypotheses  for $m(x)=o(1)$ than \eqref{gieq1.1} is provided by \eqref{gieq4.3} from Corollary \ref{gic3} but with \eqref{gieq4.4} strengthened as follows:

\begin{corollary}
\label{gic4}
If \eqref{gieq4.3} holds with distinct $t_{j}>0$ and
\begin{equation}
\label{gieq4.3.1}
 \quad (1+t_{j}^{2})^{1/2}|b_j \cos(y_j+ \arctan t_j)|<2, \quad j=1,\dots,k,
\end{equation}
then $m(x)=o(1)$.
\end{corollary}

Corollary \ref{gic2} recovers a result of Kahane and Sa\"{i}as. Their formulation from \cite{Kahane-Saias2016a} is slightly different, making use of the Liouville function. Define $\mathrm{d}L$ via its Mellin transform as
$$
\int_{1^{-}}^{\infty} x^{-s}\mathrm{d}L(x)=\frac{\zeta(2s)}{\zeta(s)} .
$$
Using elementary convolution arguments as in classical number theory, it is not hard to verify that $m(x)=o(1)$ is always equivalent to
\begin{equation}
\label{gieq5.3}
\ell(x)=\int_{1^{-}}^{x}\frac{\mathrm{d}L(u)}{u}=o(1),
\end{equation}
which is the relation that Kahane and Sa\"{i}as established in \cite{Kahane-Saias2016a,Kahane-Saias2016b} under the hypothesis \eqref{gieq1.1}. Note that for discrete generalized number systems \eqref{gieq5.3} takes the familiar form
$$
\sum_{k=0}^{\infty}\frac{\lambda(n_k)}{n_k}=0 .
$$

Finally, we point out that other sufficient conditions for $m(x)=o(1)$ are known. The validity of (\ref{gieq1.3}) has been proved in \cite[Cor. 3.1]{Debruyne-Diamond-Vindas} under a Chebyshev upper estimate hypothesis, that is,
\begin{equation}
\label{gieq5.4}
 \limsup_{x\to\infty} \frac{\Pi(x)\log x}{x}<\infty ,
\end{equation}
and the condition
\begin{equation}
\label{gieq5.5}
\int_{1}^{\infty}\frac{|N(x)-ax|}{x^{2}}\:\mathrm{d}x<\infty.
\end{equation}

We end this article with some examples that compare the different sets of hypotheses we have discussed here for $m(x)=o(1)$. In addition to these examples, observe that Corollary \ref{gic1} applies to Example \ref{giex1}, but Corollary \ref{gic2} does not because (\ref{gieq1.1}) fails for it. 

\begin{example}
\label{giex5} In this example we provide an instance of a generalized number system for which (\ref{gieq1.1}) holds (so that Corollary \ref{gic2} applies to deduce $m(x)=o(1)$), the Chebyshev upper estimate is satisfied, but the hypothesis (\ref{gieq5.5}) of  \cite[Cor. 3.1]{Debruyne-Diamond-Vindas} does not hold.

Let
$$
\mathrm{d}\Pi(u)= \frac{1-u^{-1}}{\log u}\mathrm{d}u+  \left(\frac{1-u^{-1}}{\log u}\right)^{2}\omega(u)\mathrm{d}u,
$$
where 
$\omega$ is non-increasing positive function on $[1,\infty)$ such that 
\begin{equation}
\label{gieq5.6}
\int_{2}^{\infty}\frac{\omega(x)}{x\log x}\mathrm{d}x=\infty, \quad \frac{\omega(x^{1/n})}{\omega(x)}\leq Cn^{\alpha}, \quad \mbox{and}\quad \omega(x)=o(1),
\end{equation}
with $C,\alpha>0$. (For example, $\omega(x)=1/\log\log x$ for $x\geq e^{e}$ and $\omega(x)=1$ for $x\in [1,e^{e}]$ satisfies these requirements.) Since $\omega$ is non-increasing, one readily verifies that the PNT
\[
\Pi(x)=\frac{x}{\log x}+O\left(\frac{x}{\log^{2} x}\right)
\]
holds. Thus, both (\ref{gieq5.4}) and (\ref{gieq1.1}) are satisfied; in particular, $N(x)\sim ax$ for some $a>0$, by Theorem \ref{gith1}. We have shown in \cite[Ex. 1]{d-vPNTequiv2016} the lower bound
\[
\int_{x}^{\infty}\frac{\omega(u)}{u\log^{2} u}\:\mathrm{d}u \ll \frac{ax-N(x)}{x}\:,
\]
Dividing through by $x$, integrating on $[1,\infty)$, exchanging the order of integration, and using the first condition in (\ref{gieq5.6}), we obtain that
$$
\int_{1}^{\infty}\frac{|ax-N(x)|}{x^{2}}\:\mathrm{d}x=\infty.
$$

\end{example}

\begin{example}
\label{giex6} We consider an example constructed by Kahane in \cite{kahane1998}. Let $1<a_1<a_2<\dots$ and $0<b_1<b_2<\dots$ be two sequences such that $b_{j}<a_{j+1}-a_{j}$, $\lim_{j\to\infty}b_{j}=\infty$, and
$$
\sum_{j=1}^{\infty} \frac{b_j^2}{a_j}<\infty.
$$
Define the prime measure
$$
\mathrm{d} \Pi (u)= \frac{1}{\log u}\chi_{[2,\infty)}(u)\mathrm{d}u+ \sum_{j=1}^{\infty} e^{a_{j}}\log \left(1+\frac{b_j}{a_j}\right)\delta(u-e^{a_{j}})- \frac{1}{\log u}\chi_{[e^{a_{j}},e^{a_{j}+b_{j}})}(u)\mathrm{d}u,
$$
where $\delta(u)$ is the Dirac delta measure and $\chi_{B}$ is the characteristic function of a set $B$. For this example, it is essentially shown in \cite[pp. 631--634]{kahane1998} that (\ref{gieq5.5}) holds for some $a>0$, 
$$
\limsup_{x\to\infty} \frac{\Pi(x)\log x}{x}=\infty,
$$
but $\Pi$ satisfies (\ref{gieq1.1}) (so that, once again, Corollary \ref{gic2} applies here, but \cite[Cor. 3.1]{Debruyne-Diamond-Vindas} does not).

\end{example}

\begin{example}
\label{giex7} For Example \ref{giex2}, the hypotheses of both Corollary \ref{gic1} and \cite[Cor. 3.1]{Debruyne-Diamond-Vindas} are violated, but Corollary \ref{gic4} still applies to conclude $m(x)=o(1)$. We already saw that (\ref{gieq1.1}) does not hold.  Because of \eqref{gieq4.5}, one has Chebyshev lower and upper estimates, 
$$
\frac{x}{\log x}\ll \Pi(x)\ll \frac{x}{\log x},
$$
but (\ref{gieq5.5}) fails. This can be proved by looking at the zeta function, which is given by 
$$
\zeta(s)= \frac{e^{G(s)}}{(s-1)\sqrt{1+(s-1)^{2}}}\:,
$$
where $G(s)$ is an entire function. Since $\zeta(s)$ is unbounded at  $s=1\pm i$ and (\ref{gieq5.5}) would force continuity at those points, we must have
\[
\int_{1}^{\infty}\frac{|N(x)-ax|}{x^{2}}\: \mathrm{d}x=\infty.
\]
Actually, $\log \zeta(s)-\log(1/(s-1))$ does not have a continuous extension to $\Re e\:s=1$, so that Corollary \ref{gic1} does not apply to this example either. 
\end{example}

\end{document}